\documentclass[12pt]{article}

\usepackage{amsfonts,amsmath,amsthm,latexsym,amssymb,fullpage,amscd}
\usepackage[mathscr]{eucal}
\usepackage{tikz}

 at 16 true pt
 at 13  true pt
 at 12 true pt

\newcommand{\R}{\mathbb{R}}
\newtheorem{theorem}{Theorem}[section]
\newtheorem{lemma}[theorem]{Lemma}

\title{Convexity, Fourier transforms, and lattice point discrepancy}

\author{Michael Greenblatt}
\date{\today}

\newcommand\blfootnote[1]{%
  \begingroup
  \renewcommand\thefootnote{}\footnote{#1}%
  \addtocounter{footnote}{-1}%
  \endgroup
}
\begin{document}
\maketitle
\begin{abstract} 

In a well-known paper by Bruna, Nagel and Wainger [BNW],  Fourier transform decay estimates were proved for smooth hypersurfaces of finite line type bounding a convex domain. 
In this paper, we 
generalize their results in the following ways. First, for a surface that is locally the graph of a convex real analytic function, we show that a natural analogue holds even when the
 surface in question is not of finite line type. 
Secondly, we show a result for a general surface that is locally the graph of a convex $C^2$ function, or a piece of such a surface defined through real analytic equations, that
 implies an analogous Fourier transform decay theorem in many situations where the oscillatory index is less than $1$. In such situations, for a compact surface
 the exponent provided is sharp. This result has implications for lattice point discrepancy problems, which we describe.

\end{abstract}
\blfootnote{This work was supported by a grant from the Simons Foundation.}

\section{Background and surface measure Fourier transform theorem statements.} 

\subsection {Introduction.}

We  consider Fourier transforms of hypersurface measures in $\R^{n+1}$, $n \geq 1$, where the surface is locally the graph of a convex function.
Specifically, we let $S$ be a bounded $C^2$  hypersurface in $\R^{n+1}$, such that for each $x_0 \in S$, there is a composition of a translation and a rotation, which we call $A_{x_0}$, such 
that  $A_{x_0}(x_0) = 0$, the normal to $S' = A_{x_0}(S)$ at $0$ is the vector $(0,...,0,1)$, and there is a ball $D = B(0,r_0)$ centered at $0$ such that above $D$ the surface $S'$ is 
the graph of a $C^2$ convex function $f(x_1,...,x_n)$ on a neighborhood of the closure $\bar{D}$. We localize the problem by letting $\phi(x_1,...,x_n)$ be a 
$C^1$ real-valued function supported in $D$ and we look at the Fourier transform of the measure $\mu$ on $S'$ where the surface measure is localized by the cutoff function 
$\phi(x_1,...,x_n)$ at the point on $S'$ above $(x_1,...,x_n)$. To be precise, we are looking at
\[ \hat\mu(\lambda_1,...,\lambda_{n+1}) = \int_{\R^n} e^{ - i\lambda_1 x_1 - ... -  i\lambda_nx_n -i\lambda_{n+1}f(x_1,...,x_n)} \phi(x_1,...,x_n)\,dx_1...\,dx_n\tag{1.1}\]
Note that $A_{x_0}$ is such that $f(0) = 0$ and $\nabla f(0) = 0$.

The goal here is to understand the decay of $|\hat\mu(\lambda)|$ for large $|\lambda|$. For this, it is helpful to consider $\hat\mu(\lambda)$ for $\lambda$ on rays $r(v) = 
\{tv: t \in \R\}$ for various directions $v = (v_1,...,v_{n+1})$ with $|v| = 1$. Since the ray is symmetric in $v$ and $-v$, we may assume that $v_{n+1} \geq 0$.
If $v$ is such that the angle between $v$ and the normal to the surface at every point in $S'$ above $D$ is at least some $\epsilon > 0$, then by
 repeated integrations by 
parts one has that $|\hat\mu(\lambda)| \leq C_{\epsilon, N} |\lambda|^N$ for any $N$. Thus the focus is on directions $v$ that are either normal to the surface at some point above $D$ or
close to some such direction. 

A general heuristic when examining an oscillatory integral for large values of a parameter such as $\lambda$ is that the magnitude of the oscillatory integral should be of the same order of 
magnitude as the maximal measure of the points in the domain for which the phase is within a single period. When this principle is applied to the analysis of the decay of $|\hat\mu(\lambda)|$
in a direction $v$ perpendicular to the tangent plane $T_{y}(S')$ for a point $y$ on $S'$ above the closure $\bar{D}$, this heuristic suggests the following.  Let  $\pi$ denote the
 projection onto the first $n$ variables. For any $v$ with $|v| = 1$ and $v_{n+1} \geq 0$, we  define  $s(v)$ to be the "lowest" height  in the $v$ direction achieved on the portion of $S'$ 
above $\bar{D}$. That is, we make the definition
\[ s(v) = \min\{x \cdot v: x \in S',\,\pi(x) \in \bar {D} \} \tag{1.2}\]
Then in view of the convexity of $f$, if the disk $D$ is small enough this maximal measure is bounded by the following, where $m$ denotes the surface measure on $S'$.
\[ m(\{x \in S': \pi(x) \in \bar{D},\,s(v) \leq  x \cdot v \leq  s(v) + |\lambda|^{-1}\}) \tag{1.3}\]
Hence heuristically we expect $(1.1)$ to be bounded by a constant times $(1.3)$. In addition, past experience leads one to expect that if $\phi$ is nonnegative and $s(v)$ is achieved at
 at least one point $x$ where $\phi(\pi(x)) > 0$, then $(1.3)$ should give the correct order of magnitude for $\hat{\mu}(\lambda)$. 

It follows from the classic paper [BNW] that for a smooth compact surface of finite line type bounding a convex domain (finite line type means that no line is tangent to the surface to
infinite order), then for the $t$ and $v$ in question the function $|\hat{\mu}(tv)|$ is in fact bounded by a constant independent of $v$ times  $(1.3)$ for $\lambda = tv$. In [BNW]
they make use of nonisotropic balls with a doubling property. Namely, if one defines $K(x,r) = \{x \in S': \pi(x) \in \bar{D},\,s(v) \leq  x \cdot v \leq  s(v) + r\}$, then the $K(x,r)$ are the balls for a metric space with a doubling property, which can be used to show various properties related to $\hat{\mu}$. (Actually they define the balls globally, where the same
 principle applies.) In this paper, we will see that although there are no families of balls here, one has estimates analogous to those of [BNW] in several settings.

\subsection {Surface measure Fourier transform results.} The first class of surfaces we consider are  
surfaces which are locally the graph of a convex real analytic function, but which are not necessarily of finite line type. Canonical examples of such surfaces are cones given by equations $f(\frac{x_1}{x_{n+1}}, \frac{x_2}{x_{n+1}},...,
\frac{x_n}{x_{n+1}}) = 1$, where the  $(x_1,...,x_n)$ for which $f(x_1,...,x_n) = 1$ form a  compact real analytic surface enclosing a convex domain (this surface will necessarily be of
 finite line type). Other examples include developable surfaces in three dimensions.

\noindent Our theorem in this situation is as follows.

\begin{theorem}
Suppose we are in the setting of $(1.1)$, where $\phi(x)$ is $C^{n+1}$ and supported in D, and $f(x)$ is convex, real analytic, and not identically zero. Then for any $v$ with
$|v| = 1$ and  $v_{n+1} \geq 0$, and any $t \neq  0$ we have an estimate
\[|\hat{\mu}(t v)| \leq C \,m(\{x \in S': \pi(x) \in\bar{D},\,s(v) \leq  x \cdot v \leq  s(v) + |t|^{-1}\}) \tag{1.4}\]
Here $m$ denotes Euclidean surface measure and the constant $C$ depends only on $\phi$ and $f$. 
\end{theorem}
While $(1.4)$ holds for all $t$, it is only of interest for large $|t|$ since the result for $|t| < B$ for a fixed $B$ will follow immediately simply by taking absolute values inside the integrand and 
integrating. The same is true for the other theorems of this paper. 

The natural geometric interpretation of the measure on the right-hand side of $(1.4)$ is the volume of the intersection of the portion of $S'$ over $D'$ with the slab of 
width $|t|^{-1}$ perpendicular to $v$ whose lower side intersects this portion of $S'$ at the lowest possible point in the $v$ direction. There are situations where for a
 given $v$ this does not give the best possible rate of decay, but when $v$ is perpendicular to the tangent plane of $S'$ at a point above $D$ it typically does. 

At first glance, one might wonder if Theorem 1.1 can be proven by taking strictly convex cross sections of the surface parallel to $v$, of some appropriate dimension, then applying [BNW] to the cross sections, and then integrating the resulting estimates in remaining directions corresponding to the intersection of $S'$ with its tangent plane at points $x$ minimizing $x \cdot v$.  But one runs into a few issues if one tries to do this. For one, the dimension of this intersection may vary between different such points $x$, making it unclear what dimension these cross sections 
should be. Furthermore, one needs the resulting estimates to be uniform under perturbations of $v$, many of which will lead to situations where $x \cdot v$ is minimized on the boundary
 of $S'$. Thus this sort of approach is not easy to put into practice.

In a sense, Lemmas 4.1 and 4.3 will provide a substitute for real analytic functions for the finite type condition of [BNW], so that an analogous theorem can be proven. If one has a compact 
 real analytic surface $S$ enclosing a convex domain (which will necessarily be of finite line type), then one can recover 
the Fourier transform bounds of [BNW] for $S$ by using a partition of unity to write the surface measure Fourier transform as a finite sum of integrals of the form $(1.1)$, after appropriate
translations and rotations. For a given piece, one applies Theorem 1.1 in directions that are normal or nearly normal to the tangent planes of the surface piece, and simply integrates by 
parts repeatedly in the oscillatory integral for directions that are not nearly normal to the surface piece. Adding the results will then give the Fourier transform bounds of [BNW] for that surface.

Our second theorem holds for all surfaces that are graphs of convex functions that are at least $C^2$, and where $\phi$ is at least $C^1$, as well as graphs of portions of such surfaces 
carved out by real analytic functions. Specifically, we assume $f(x)$ is $C^2$ and convex, $\phi(x)$ is $C^1$, and 
 $g_1(x),...,g_m(x)$ are real analytic functions defined on a neighborhood of $\bar{D}$, none identically zero.
 We define $I(\lambda)$ to be the Fourier transform of a surface measure carved out by the
the $g_i$ on the surface over $D$ whose graph is $f(x)$. Specifically, we define
\[I(\lambda) = \int_{\{x \in \R^n:\,g_i(x) < \,0 {\rm\,\, for\,\, all\,\,} i \}} e^{ - i\lambda_1 x_1 - ... -  i\lambda_nx_n -i\lambda_{n+1}f(x_1,...,x_n)} \phi(x_1,...,x_n)\,dx_1...\,dx_n\tag{1.5}\]
Note that by choosing $g_2 = -g_1 + 1$, the case where there are no $g_i(x)$ is included in the above. Our theorem for $I(\lambda)$ is as follows.

\begin{theorem}
For any $v = (v_1,...,v_{n+1})$ with $|v| = 1$ and $v_{n+1} \geq 0$, and any $t \neq 0$ we have an estimate
\[|I(tv)| \leq C m(\{x \in S': \pi(x) \in\bar{D},\,s(v) \leq  x \cdot v \leq  s(v) + |t|^{-1}\}) \]
\[ + C  \sum_{j=1}^{\infty}2^{-j}\,m(\{x \in S': \pi(x) \in\bar{D},s(v) + 2^{j-1}|t|^{-1} \leq  x \cdot v \leq  s(v) + 2^j |t|^{-1}\})\tag{1.6}\]
Here $m$ denotes Euclidean surface measure and the constant $C$ depends only on $\phi$, $f$, and the $g_i$.
\end{theorem}

The terms in $(1.6)$ can be interpreted as volumes of intersections of slabs perpendicular to $v$ with the portion of $S'$ above $\bar{D}$, similarly to Theorem 1.1. Note that the pieces 
of $S'$ carved out by the functions $g_i(x)$ are not what are used in defining the slabs, but rather the original surface $S'$.

To help further understand the meaning of Theorem 1.2, suppose we are in a situation where for some $v$ with $|v| = 1$ and $v_{n+1} \geq 0$, there is a $0 < \alpha < 1$ such that
 for all $t \neq 0$ we have an estimate 
\[m(\{x \in S': \pi(x) \in\bar{D},\,s(v) \leq  x \cdot v \leq  s(v) + |t|^{-1}\}) \leq C|t|^{-\alpha} \tag{1.7}\]
Then by adding this estimate over all $j$, $(1.6)$ implies that $|I(tv)| \leq C'|t|^{-\alpha}$ for all $t$. If the estimate $(1.7)$ holds for all $t \neq 0 $ and all directions $v$ with $v_{n+1} \geq 0$,
 then we will similarly have a uniform estimate $|I(tv)| \leq C'|t|^{-\alpha}$  for all such $t$ and $v$. In this way  if the optimal $\alpha$ for which one has a uniform estimate $(1.7)$ is in the 
range $0 < \alpha < 1$, one retains this rate of decay in the surface measure Fourier transform $I(\lambda)$.  We record this fact in the following theorem.

\begin{theorem} Suppose we are in the setting of Theorem 1.2 and for some $C > 0$ and some $0 < \alpha < 1$ the estimate $(1.7)$ holds for all $t \neq 0$ and  all $v$ with $|v| = 1$ and
$v_{n+1} \geq 0$. Then there exists a constant $C'$ such that one has the estimate $|I(\lambda)| \leq C'|\lambda|^{-\alpha}$ for all $\lambda \neq 0$. 
\end{theorem}

In the setting of Theorem 1.2, the surface may be flat to 
infinite order at some point $x$. In a direction $v$ perpendicular to $T_x(S)$, the rate of decay of the left hand side of $(1.6)$ may be slower than any $|t|^{-\alpha}$ for $\alpha > 0$. 
Nonetheless, in some such scenarios one may still sum the series in $(1.6)$ and obtain a result bounded by a constant times the first term in $(1.6)$. In such a way, one can still bound
$|\hat{\mu}(tv)|$ in terms of this first term, much as in Theorem 1.1.

Theorem 1.3 can be readily globalized when there are no $g_i(x)$ as follows.
 Suppose $S$ is a compact $C^2$ surface bounding a convex domain. Let $\nu$ denote Euclidean surface measure on $S$ and let  $\psi(x)$ be a $C^1$ function defined on a 
neighborhood of $S$ in $\R^{n+1}$. Then one can bound
$|\widehat{\psi \nu}(\lambda)|$ by using a partition of unity and then applying Theorem 1.3 to each term. Because one can get different $s(v)$ corresponding to each term, in order to have 
a reasonable theorem statement we proceed as follows. For a given direction $v$ and $t \neq 0$ we define $a(v,t)$ by
\[ a(v,t) = \sup_{s \in \R} m(\{x \in S: s  \leq  x \cdot v \leq  s + |t|^{-1}\}) \tag{1.8}\]
Thus $a(v,t)$ is the maximal value of the measure of intersection of $S$ with any "slab" perpendicular to $v$ of width $|t|^{-1}$. 
Note that with any  partition of unity, any term $m(\{x \in S': \pi(x) \in\bar{D},\,s(v) \leq  x \cdot v \leq  s(v) + |t|^{-1}\})$ showing up from an application of Theorem 1.3 will be bounded 
by $a(v,t)$. Thus if we have the  bound $a(v,t) \leq C|t|^{-\alpha}$ holding uniformly in $v$ for some $0 < \alpha < 1$, then the sum of all the terms in $(1.6)$ will similarly be bounded
 by some $C'|t|^{-\alpha}$. In other words, by adding over terms of the partition of unity, Theorem 1.3 implies that $|\widehat{\psi \nu}(\lambda)| \leq C'|\lambda|^{-\alpha}$ here. This
 leads to the following sharp result.

\begin{theorem} Let $S$ be a compact $C^2$ surface bounding a convex domain.  Let $\nu$ denote Euclidean surface measure on $S$ and let  $\psi(x)$ be a $C^1$ function defined on a 
neighborhood of $S$ in $\R^{n+1}$.  Suppose $0 < \alpha < 1$. 

\begin{enumerate} 

\item If there exists a $C > 0$ for which  $a(v,t) \leq C|t|^{-\alpha}$ for all directions $v$ and all $t \neq 0$, then 
there exists a  constant $C'$ depending on $S$, $\psi$, and $\alpha$ such that $|\widehat{\psi \nu}(\lambda)| \leq C'|\lambda|^{-\alpha}$ for all $\lambda \neq 0$.
\item Conversely, suppose $\psi(x) > 0$ for all $x \in S$. If there is a constant $C'$  such that $|\widehat{\psi \nu}(\lambda)| \leq C'|\lambda|^{-\alpha}$ for all $\lambda \neq 0$, then there exists a $C > 0$ depending on $S$ and $\alpha$ for which  $a(v,t) \leq C|t|^{-\alpha}$ for all directions $v$ and all $t \neq 0$.

\end{enumerate}
\end{theorem}

\begin{proof} Part 1 follows from Theorem 1.3 as described above. To prove part 2, we use a well-known tool in the field that converts oscillatory
 integral estimates into sublevel set estimates. We use the following formulation from [Gr], which originates in  [CaCWr].

\begin{lemma}
Let $(Y,\mu)$ be any finite measure space, $0<\delta<1$ and $f:Y\rightarrow\mathbb{R}$ a measurable function such that for all real non-zero $\lambda$, we have
\[\left|\int_Ye^{i\lambda f(y)}\,d\mu(y)\right|\leq A|\lambda|^{-\delta} \tag{1.9}\]
Then for each $c\in\mathbb{R}$ and $\epsilon > 0$ we have
\[\mu(\{y\in Y:|f(y)-c|\leq\varepsilon\})\leq C_\delta A\varepsilon^\delta \tag{1.10}\]
where $C_\delta$ depends only on $\delta$.
\end{lemma}
To prove part 2 of Theorem 1.4, we apply Lemma 1.5 with $Y= \R$, $f(y) = y$, $\delta = \alpha$, and 
$\mu(-\infty, y) = {\displaystyle \int_{\{x \,\in\, S:\, x\cdot v \,\leq\, y\}} \psi(x)\, d\nu}$. In this situation, the oscillatory integral of $(1.9)$ is equal to $\widehat{\psi \nu}(\lambda v)$.
 If the estimate $(1.9)$ holds for all $v$ and $\lambda$ then by Lemma 1.5 one also has an estimate $a(v,t) \leq C|t|^{-\alpha}$ for all $v$ and $t$. This completes the proof of the theorem.

\end{proof}

It is worth pointing out that by using Lemma 1.5 similarly to the above proof, one has that a Fourier transform decay rate exponent  provided by the main result of [BNW] is also sharp when
it is less than 1. In many situations it will be sharp even when greater than or equal to 1, but this would have to be shown differently.

Quantities related to $a(v,t)$ have appeared in other results connecting oscillatory integrals to sublevel set measures, such as the recent paper
[BaGuZhZo]. In addition, conditions somewhat resembling those of Theorems 1.3 and 1.4 have appeared in related contexts, such as [G2] and [ISa].
There also has been quite a bit of other work on Fourier transforms of surface measures of surfaces that are locally graphs of convex functions. We mention 
 [BakMVW] [BrHoI] [CoDiMaMu] [Gr] [Gre] as examples especially pertinent to the topic of this paper. We also mention the book [IL] which contains many results in this area and further references. 

\section{Lattice point discrepancy.} 

Suppose $S$ is a compact  $C^2$ surface bounding a convex domain $S_0$ containing the origin. For $k > 1$, let $kS_0$ denote the dilated surface $\{kx: x \in S\}$, and let
 $N(k)$ denote the number of lattice points on or inside $kS_0$. Then one has $N(k) \sim k^{n+1}m(S_0)$ for large $k$, and a straightforward geometric argument gives that for some constant
$C$ one also has an estimate $|N(k) -  k^{n+1}m(S_0)| \leq Ck^n$. If $S_0$ were a polyhedron instead of having $C^2$ boundary, then it is not hard to show that the flatness of the
sides of $S_0$ ensures that the exponent $n$ in $Ck^n$ is best possible.

It turns out that when $S$ is curved enough for the conditions of part 1 of Theorem 1.4 to hold for some $\alpha > 0$ then one gets a better exponent than $n$. Specifically, there are 
well known methods (see 
p.383-384 of [ShS]) that give the following. Suppose $A$ is a compact $C^2$ surface in $\R^{n+1}$ bounding an open set $V$ such that there exist positive constants $c$ and $\epsilon_0$
such that if $\epsilon < \epsilon_0$, whenever $x \in V$ and $|y| < \epsilon$ one has $x + y \in (1 + c\epsilon)V$. Then if the Euclidean surface measure on $A$, which we denote by $\rho$,
satisfies $|\hat{\rho}(\xi)| \leq C|\xi|^{-\beta}$ for some $\beta, C > 0$, then the lattice point discrepancy corresponding to $V$ satisfies $|N(k) - k^{n+1}m(V)| \leq
C'k^{n - {\beta \over n + 1 - \beta}}$ for some constant $C'$.

When the conditions of part 1 of Theorem 1.4 hold, the surfaces $S$ at hand satisfy the above conditions with $\beta = \alpha$; the fact that $S_0$ is convex with $C^2$ boundary
 ensures that the requisite condition on $S_0$ holds. Thus Theorem 1.4 immediately leads to the following consequence.

\begin{theorem}
Suppose $S$ is a compact  $C^2$ surface bounding a convex domain $S_0$ containing the origin.  Suppose for some $0 < \alpha < 1$ and some $C > 0$ we have $|a(v,t)| \leq C|t|^{-\alpha}$
 for all directions $v$ and all $t \neq 0$. Then for some constant $C'$ one has the lattice point discrepancy bound $|N(k) -  k^{n+1}m(S_0)| \leq C'k^{n - {\alpha \over n + 1 - \alpha}}$
\end{theorem}

Finding the optimal exponent for the lattice discrepancy for a given domain can be very difficult in general. The famous unsolved Gauss circle problem is to show that for the disk in two
dimensions one has the bound  $N(k) - \pi  k^2 = O(k^{\frac{1}{2} + \epsilon})$ for any $\epsilon > 0$; it was shown by Hardy in [H1] [H2] that one does not have an 
$O(k^{1 \over 2})$ bound. The surface Fourier transform method used above gives an exponent of $\frac{2}{3}$. The current best known exponent is $.628966...$, due to 
Li and Yang [LiY]. 

For spheres, the problem becomes less difficult as the dimension increases. While the best possible exponent is unknown for spheres in three dimensions, if $n \geq 3$ one can show that
$N(k) -  k^{n+1}m(V) = O(s^{n-1})$ and the exponent $n-1$ is best possible. We refer to [Kr] for more information about lattice point discrepancy for spheres.

There has been a lot of work on lattice point discrepancy problems for more general convex domains. If the boundary of the domain $V$ has nonzero Gaussian curvature and is sufficiently
smooth, then the Fourier transform method described above gives the bound $N(k) - k^{n+1} m(V) = O(k^{n - 1 + {2 \over n + 2}})$, as was proven by Hlawka [Hl1] [Hl2]. Various 
improvements have been made to this bound over time. There have also been various papers where the Gaussian curvature condition or the smoothness condition has been relaxed, such as the
references [BrIT] [ISaSe1] [ISaSe2] [ISaSe3] [R1] [R2]. These papers use harmonic analysis techniques in the proofs. 

\section{Some examples.}

\noindent {\bf Example 1}. 

We consider conic surfaces given by equations $f(\frac{x_1}{x_{n+1}}, \frac{x_2}{x_{n+1}},...,
\frac{x_n}{x_{n+1}}) = 1$, where the points $(x_1,...,x_n)$ for which $f(x_1,...,x_n) = 1$ form a compact real analytic surface enclosing a convex domain (this surface will necessarily be of
 finite line type). Note that such conic surfaces are not of finite line type. 
While one can analyze this example directly, it is also a good illustration of the statements of Theorems 1.1 and 1.2.  Let $\rho$ be the Euclidean 
surface measure corresponding to such a surface, and let $\psi(x)$ be a $C^{n + 1}$ function on $\R^{n+1}$ supported in the set $1 < x_{n+1} < 2$. Then the measure
$\psi(x)\rho$ falls under the hypotheses of Theorem 1.1; one can use a partition of unity to reduce the situation to pieces which, after a suitable translation and rotation, become of the 
form $(1.1)$. 

We examine what Theorem 1.1 says in this situation. If $v$ makes an angle less than say $\frac{\pi}{8}$ with the $x_{n+1}$ axis, one can get arbitrarily fast decay rate for the integral $(1.1)$ 
simply by repeated integrations by parts in directions parallel to $S$. On the other hand, if $v$ makes an angle $\frac{\pi}{8}$ or more with the $x_{n+1}$ axis, one may look at a slab
given by Theorem 1.1 of width $|t|^{-1}$  perpendicular to the $v$ direction, and that will give a rate of decay in the $v$ direction. If $v$ is perpendicular to the tangent plane to 
the cone, one may get a decay rate as fast as $C|\lambda|^{-{n - 1 \over 2}}$, which occurs when the Hessian determinant is of maximal rank $n-1$.

Next, we examine the consequences of Theorem 1.2 for this example. Suppose once again
$v$ makes an angle greater than or equal to $\frac{\pi}{8}$ with the $x_{n+1}$ axis. We look at the intersection of the $n+1$-dimensional  slab $\{x \in S: s  \leq  x \cdot v \leq  s + |t|^{-1}\}$ 
with the cross section of $S$ at a given height $x_{n+1} = c$, which is the  intersection of an $n$-dimensional slab of $\{x: x_{n+1} = c\}$
 with this cross section of $S$, again  with width comparable to  $|t|^{-1}$. Thus if one wants an overall decay rate of the 
form $|\widehat{\psi(x)\rho}(\lambda)| \leq C|\lambda|^{-\alpha}$ for some $0 < \alpha < 1$, one can seek an $\alpha$ for which $a(v,t) \leq C|t|^{-\alpha}$  holds for all $t$ and $v$ for 
the surface $f(x_1,...,x_n) = 1$, since 
the same exponent will work for each surface $f({x_1 \over c},...,{x_n \over c}) = 1$. Then by Theorem 1.2 one will get $|\widehat{\psi(x)\rho}(tv)| \leq C|t|^{-\alpha}$ for all 
$t$ and all $v$ with $|v| = 1$  making an angle
$\frac{\pi}{8}$ or more with the $x_n$ axis. Combining with the case of the angle being less than $\frac{\pi}{8}$, we see that overall one has a Fourier transform decay rate
of $|\widehat{\psi(x)\rho}(\lambda)| \leq C|\lambda|^{-\alpha}$.

\noindent {\bf Example 2.} 

Let $S_1,...,S_k$ be distinct  compact  real analytic surfaces enclosing convex regions $V_1,...,V_k$ respectively. Let
 $V = \cup_{i=1}^k V_i$ and let $S$ be
the boundary of $V$. Then $S$  the finite union of pieces of different $S_i$, each of which falls under Theorems 1.2 and 1.3. Let $\sigma$ be the Euclidean surface measure 
on $S$ and $\sigma_i$ the Euclidean surface measure on $S_i$. Suppose $\alpha > 0$ is such
that for each $i$, Theorem 1.3 gives a bound of $\hat{\sigma_i}(\lambda) \leq C'|\lambda|^{-\alpha}$. Then Theorem 1.3 implies one has the same bound of $C'|\lambda|^{-\alpha}$
for the Fourier transform of the surface measure of the piece of $S$ deriving from $S_i$. Adding this up over all $i$, we see that $|\hat{\sigma}(\lambda)| \leq C|\lambda|^{-\alpha}$ as well. 

Note that the above argument will also work for an $S$ that is the boundary of a finite intersections of $V_i$, or more generally finite unions of finite intersections of such $V_i$. 
One can sometimes improve the estimates obtained by examining the piece of $S_i$ occurring in a given $S$; one only needs the best exponent for a neighborhood in $S_i$ of this piece.

\noindent {\bf Example 3.} 

Suppose $S$ is any $C^2$ curve in $\R^2$ enclosing a bounded convex region, and let $\nu$ denote its surface measure.
 Let $P(x_1)$ denote the phase function appearing in $(1.5)$, viewed as a function of $x_1$. Since $| P'(x_1)|$ is nondecreasing as $x_1$ moves away
 from a point $x_1^*$ with $P'(x_1^*) = 0$, when such an $x_1^*$ exists the terms in the sum of $(1.6)$ are nonincreasing in $j$ and bounded by the term outside the sum. 

Given a direction $v$, let $ s(v) = \min\{x \cdot v: x \in S \}$ and $ s^*(v) = \max\{x \cdot v: x \in S \}$. In view of the above, by applying Theorem 1.2 near $x$ for which 
$s(v)$ and $s^*(v)$ are achieved, one obtains the following estimate.
\[ |\hat{\nu}(tv)| \leq C\ln (2 + |t|)\, m(\{x \in S: s(v) \leq x \cdot v \leq s(v) + |t|^{-1}\}) \]
\[+\, C\ln (2 + |t|) \,m(\{x \in S: s^*(v) - |t|^{-1} \leq x \cdot v \leq s^*(v)\}) \tag{3.1} \] 
It turns out that if $S$ is a $C^{\infty}$ curve one can remove the $\ln (2 + |t|)$ factors in $(3.1)$. This was shown in [BNW].

\section{Proofs of Theorems 1.1 and 1.2.}

\subsection{Some lemmas used in the proof of Theorem 1.1.}

In this section, we will make use of a couple of results from [G1] about real analytic functions (which are closely related to some results in [Mi]). They are as follows.

\begin{lemma}(Theorem 2.1 of [G1].) Suppose $g(x_1,...,x_n)$ is a real analytic function defined on a neighborhood of the origin, not identically zero. Then there is an $n-1$-dimensional ball
 $B_{n-1}(0,\eta)$ and a $k \geq 0$ such that for each $(x_1,...,x_{n-1})$ in $B_{n-1}(0,\eta)$ either $g(x_1,...,x_n) = 0$ for all $|x_n| < \eta$ or
  there is a $0 \leq l \leq k$, which may  depend on $(x_1,...,x_{n-1})$, such that  for all $|x_n| < \eta $ one has
\[0 <  {1 \over 2} |\partial_{x_n}^l g(x_1,...,x_{n-1},0)| <  |\partial_{x_n}^l g(x_1,...,x_n)| <  2|\partial_{x_n}^l g(x_1,...,x_{n-1},0)|  \tag{4.1}\]
 The set of $(x_1,...,x_{n-1})$ for which $g(x_1,...,x_n) = 0$ for all $|x_n| < \eta$ has measure zero.
\end{lemma}

\begin{lemma}(Corollary 2.1.2 of [G1].) Suppose $f_1(x_1,...,x_n)$,...,$f_l(x_1,...,x_n)$ are real analytic functions on a neighborhood of the origin, none identically zero. Then there
is an $n-1$ dimensional ball $B_{n-1}(0,\eta)$ and a positive integer $p$ such that for each $s_1,...,s_l$ and each $(x_1,...,x_{n-1}) \in B_{n-1}(0,\eta)$,
 the set $\{x_n: |x_n| < \eta $ and $f_i(x_1,...,x_n) < s_i$ for each $i\}$ consists of at most $p$ intervals.
\end{lemma}

 Lemma 4.1 will be used in the proof of Theorem 1.1 in conjunction with the following result from [PS].

\begin{lemma} (Lemma 1 of [PS].) Suppose $F \in C^N[\alpha,\beta]$, with $N \geq 1$, such that for a constant $C\geq 0$ one has
\[\sup_{\alpha \leq x \leq \beta} |F^{(N)}(x)| \leq C\inf_{\alpha \leq x \leq \beta} |F^{(N)}(x)| \tag{4.2}\] 
Then if $I$ is any subinterval of $[\alpha,\beta]$ of length $\delta$ and $I^*$ denotes its concentric double (as a subset of $[\alpha,\beta]$), we have 
\[\sup_{x\in I^*} |F(x)| \leq a \sup_{x\in I} |F(x)|\tag{4.3a}\]
\[\sup_{x\in I^*} |F'(x)| \leq a\delta^{-1}\sup_{x\in I} |F(x)|\tag{4.3b}\]
Here the constant $a$ depends only on $N$ and $C$. 
\end{lemma}

\noindent We now proceed to the proofs of Theorem 1.1 and 1.2, starting with Theorem 1.1.

\subsection {The proof of Theorem 1.1.}

In equation $(1.1)$, we write $v = \frac{\lambda}{|\lambda|}$ and $P(x_1,...,x_n) =  (x_1,...,x_n, f(x_1,...,x_n))$. Then $(1.1)$ can be rewritten as
\[ \hat{\mu}(\lambda) = \int_{\R^n} e^{-i|\lambda| P(x) \cdot v}\phi(x)\,dx \tag{4.4}\]
Since $|I(\lambda)| = |I(-\lambda)|$, without loss of generality we may assume $v = (v_1,...,v_{n+1})$ satisfies $v_{n+1} \geq 0$ as in the statement of the theorem. 
For such a $v$, one has $s(v)  = \min_{x \in \bar{D}} (P(x) \cdot v)$. Choose any $x_0$ such that $s(v) = P(x_0) \cdot v$. Then $(4.4)$ gives
\[ |\hat{\mu}(\lambda)| = \bigg|\int_{\R^n} e^{-i|\lambda| (P(x) \cdot v -  P(x_0) \cdot v)}\,\phi(x)\,dx\bigg| \tag{4.5}\]
We integrate the integral of $(4.5)$ in a polar coordinate system centered at $x_0$. Namely we rewrite this integral as 
\[ \hat{\mu}(\lambda) = c_n\int_{S^{n-1}}\int_0^{\infty} e^{-i|\lambda| (P(x_0 + r \omega ) \cdot v -  P(x_0) \cdot v)}\,r^{n-1}\phi(x_0 + r\omega)\,dr\, d\omega\tag{4.6}\]
Let $\psi(x)$ be a nonnegative smooth bump function on $\R$ supported on $[-1,1]$ with $\psi(x) = 1$ on $[-{1 \over 2},{1 \over 2}]$. Let $\psi_1(x) = 1 - \psi(x)$. We write $(4.6)$ as
 $I_1 + I_2$, where
\[I_1 =  c_n\int_{S^{n-1}}\int_0^{\infty} e^{-i|\lambda| (P(x_0 + r \omega ) \cdot v -  P(x_0) \cdot v)}\,r^{n-1}\phi(x_0 + r\omega)\psi(|\lambda|(P(x_0 + r \omega ) \cdot v -  P(x_0) \cdot v))\,dr\, d\omega \tag{4.7a}\]
\[I_2 =  c_n\int_{S^{n-1}}\int_0^{\infty} e^{-i|\lambda| (P(x_0 + r \omega ) \cdot v -  P(x_0) \cdot v)}\,r^{n-1}\phi(x_0 + r\omega)\psi_1(|\lambda|(P(x_0 + r \omega ) \cdot v -  P(x_0) \cdot v))\,dr\, d\omega \tag{4.7b}\]
Note that by simply taking absolute values in the integral $(4.7a)$, integrating, and then going back into rectangular coordinates, one sees that $(4.7a)$  is bounded by the right-hand side of
$(1.4)$. Thus our concern here is bounding $(4.7b)$. 

To simplify notation we let $P_{x_0,\omega,v}(r) = P(x_0 + r \omega ) \cdot v -  P(x_0) \cdot v$, so that the inside integral of $(4.7b)$ can be just written as 
\[\int_0^{\infty} e^{-i|\lambda| P_{x_0,\omega,v}(r) }\,r^{n-1}\phi(x_0 + r\omega)\psi_1(|\lambda P_{x_0,\omega,v}(r)| )\,dr\tag{4.8}\]

We examine $(4.8)$. Since $x_0$ was chosen so that $P(x_0) \cdot v $ achieves the minimal possible value of $P(x) \cdot v$ for $x \in \bar{D}$, one has that $P_{x_0,\omega,v}(0) = 0$
 and $\partial_r P_{x_0,\omega,v}(0) \geq 0$ (at least for $\omega$ for which there is a nonzero $r$ integral.)  Since
$f$ is convex, we also have that $\partial_{rr} P_{x_0,\omega,v}(r) \geq 0$ for all $r$ in $(4.8)$, so that $\partial_r P_{x_0,\omega,v}(r) \geq \partial_r P_{x_0,\omega,v}(0) 
\geq 0$ for all  $r$. Furthermore, since $f$ is real analytic here, either $\partial_r P_{x_0,\omega,v}(r)$ is identically zero in $r$  (in which case $P_{x_0,\omega,v}(r)$ is also  identically 
zero in $r$), or $\partial_r P_{x_0,\omega,v}(r) > 0$
for all $r > 0$. Because of the $\psi_1(|\lambda P_{x_0,\omega,v}(r)| )$ factor in $(4.8)$, in the former situations $(4.8)$ is just zero, so for the purposes of our analysis we may exclude such
cases and work under the assumption that $\partial_r P_{x_0,\omega,v}(r) > 0$ for all $r > 0$, which in turn implies $P_{x_0,\omega,v}(r) > 0$ for all $r > 0$.

\noindent Also, since $\partial_{rr} P_{x_0,\omega,v}(r) \geq 0$, and therefore $\partial_r P_{x_0,\omega,v}(s)$ is increasing, we have
\[ P_{x_0,\omega,v}(r) = \int_0^r  \partial_r P_{x_0,\omega,v}(s)\,ds \]
\[ \leq r\partial_r P_{x_0,\omega,v}(r)\]
We will normally use this in the form 
\[{1 \over |\lambda r\, \partial_r P_{x_0,\omega,v}(r)| } \leq \frac{1} {|\lambda P_{x_0,\omega,v}(r)|}\tag{4.9}\]
Next, we apply Lemma 4.1 to $\partial_r^{n+2} P_{x_0,\omega,v}(r)$, where one treats $r$ as the $x_n$ variable in the statement of the  lemma, and the $x_0,$ $\omega$, and $v$ 
variables as the remaining $n-1$ variables. While Lemma 4.1 is a local statement, since we are working on a compact set in all variables, it immediately implies a corresponding statement over all 
$(x_0, \omega, v, r)$. Namely there is some $k \geq n + 2$ and $\delta > 0$  such given any $(\omega, x_0, v, r)$, there is a $C_{x_0, \omega, v}$ (which may be zero) and an 
$n + 2 \leq l_{x_0, \omega, v} \leq k$ such that for $|r' - r| < \delta$ one has
\[ {1 \over 2}C_{x_0, \omega, v} \leq |\partial_r^{l_{x_0, \omega, v}} P_{x_0,\omega,v}(r)| \leq C_{x_0, \omega, v} \tag{4.10}\]
We now apply Lemma 4.3 in conjunction with $(4.10)$. By applying $(4.3b)$ repeatedly, we get that there is a constant $b$ such that for all $1 \leq j \leq n + 2$ and all $r > 0$ one has
\[\sup_{t\in [0,r]} |\partial_r^j P_{x_0,\omega,v}(r)| \leq br^{-j+1}\sup_{t\in [0,r]} |\partial_r P_{x_0,\omega,v}(r)|\tag{4.11}\]
Since $\partial_r P_{x_0,\omega,v}(r)$ is nonnegative and increasing, this implies that
\[|\partial_r^j P_{x_0,\omega,v}(r)| \leq br^{-j+1} \partial_r P_{x_0,\omega,v}(r) \tag{4.12}\]
We now proceed to the analysis of $(4.8)$. Since $\partial_r P_{x_0,\omega,v}(r)$ is positive and nondecreasing for $r > 0$, $\partial_r P_{x_0,\omega,v}(r)$ is bounded below away from zero on the support of the
$\psi_1(|\lambda P_{x_0,\omega,v}(r)| )$ factor appearing in $(4.8)$. As a result, we may integrate by parts as follows, with no endpoint terms appearing. We write 
\[e^{-i|\lambda| P_{x_0,\omega,v}(r) } =  -i|\lambda| \partial_r P_{x_0,\omega,v}(r) e^{-i|\lambda| P_{x_0,\omega,v}(r) } \times (-i|\lambda| \partial_r P_{x_0,\omega,v}(r))^{-1}\]
 In $(4.8)$ we integrate  $-i|\lambda| \partial_r P_{x_0,\omega,v}(r) e^{-i|\lambda| P_{x_0,\omega,v}(r) }$ back to $e^{-i|\lambda| P_{x_0,\omega,v}(r) }$, and differentiate 
$ (-i|\lambda| \partial_r P_{x_0,\omega,v}(r))^{-1}$ times the remaining factors. We perform this integration by parts a total of $n + 1$ times.

The idea is that after $k$ integrations by parts, the integrand incurs a factor bounded by a constant times $|\lambda P_{x_0,\omega,v}(r)|^{-k}$. To see why this is
the case, we examine all the possibile factors the $r$ derivative may land in an integration by parts, and we will see that each time the integrand is multiplied a factor bounded by a constant
times $|\lambda P_{x_0,\omega,v}(r)|^{-1}$. We start with the situation where the derivative lands on a $(\partial_r P_{x_0,\omega,v}(r))^l$ appearing in the denominator of a ratio of derivatives of
$ P_{x_0,\omega,v}(r)$, such as the original $(-i|\lambda| \partial_r P_{x_0,\omega,v}(r))^{-1}$ in the first integration by parts. Then the $(\partial_r P_{x_0,\omega,v}(r))^{-l}$
becomes $\partial_{rr}  P_{x_0,\omega,v}(r)(\partial_r P_{x_0,\omega,v}(r))^{-l - 1}$, so that we have incurred a factor bounded by $C|\partial_{rr}  P_{x_0,\omega,v}(r)(\partial_r P_{x_0,\omega,v}(r))^{-1}|$. By $(4.12)$ this  is at most $C'r^{-1}$. When combined with the preexisting $(-i|\lambda| \partial_r P_{x_0,\omega,v}(r))^{-1}$ factor in the 
integration by parts, we see we have a factor bounded by $C' |\lambda r\, \partial_r P_{x_0,\omega,v}(r)|^{-1}$, which by $(4.9)$ is bounded by the desired 
 constant times $|\lambda P_{x_0,\omega,v}(r)|^{-1}$, the factor we seek.

Next, we consider the case where the $r$ derivative lands on some $\partial_r^k P_{x_0,\omega,v}(r)$ appearing in the numerator of a ratio of derivatives of $P_{x_0,\omega,v}(r)$. Then 
the  $\partial_r^k P_{x_0,\omega,v}(r)$ becomes a  $\partial_r^{k+1} P_{x_0,\omega,v}(r)$. Whereas before $|\partial_r^k P_{x_0,\omega,v}(r)|$ was being estimated using $(4.12)$ with
$j = k$, we now estimate  $|\partial_r^{k+1} P_{x_0,\omega,v}(r)|$ with $(4.12)$ with $j = k+1$. Thus we incur an additional factor of $C r^{-1}$. Hence again when combined with
 the preexisting $(-i|\lambda| \partial_r P_{x_0,\omega,v}(r))^{-1}$ factor, we obtain the desired factor of a constant times $|\lambda P_{x_0,\omega,v}(r)|^{-1}$.

Next, we look at when the $r$ derivative lands on the $\psi_1(|\lambda P_{x_0,\omega,v}(r)| )$ factor. We obtain a factor bounded by a constant times 
$|\lambda| \partial_r P_{x_0,\omega,v}(r)$, which is exactly cancelled out by the preexisting $(-i|\lambda| \partial_r P_{x_0,\omega,v}(r))^{-1}$ factor. Thus it would seem we just incur a 
factor bounded by a constant here. However, the derivative turns the $\psi_1(|\lambda P_{x_0,\omega,v}(r)| )$ into a $\psi_1'(|\lambda P_{x_0,\omega,v}(r)| )$ and $\psi_1'$ is compactly
supported. Hence we may simply insert $1 \leq C |\lambda P_{x_0,\omega,v}(r)|^{-1}$ here to obtain the desired factor of a constant times
 $|\lambda P_{x_0,\omega,v}(r)|^{-1}$.

We move on to the case where the derivative lands on $r^{n-1}$. This gives a factor of $Cr^{-1}$, which like in the earlier cases gives us the factor we seek. Lastly, we consider the case
where the derivative lands on the $\phi(x_0 + r\omega)$ factor. Here we simply incur a bounded factor, which will be better than the $Cr^{-1}$ we need.

We have now considered all possibilities and we see that with each integration by parts incurs a factor bounded by $C |\lambda P_{x_0,\omega,v}(r)|^{-1}$. Thus if we 
integrate by parts in this fashion $n + 1$ times, we see that the inside integral $(4.8)$ is bounded by the following, where $K$ denotes the interval of integration in this integral.
\[ C \int_K  r^{n-1}\psi_1(|\lambda P_{x_0,\omega,v}(r)| )\frac{1} {|\lambda P_{x_0,\omega,v}(r)|^{n+1}}\,dr\tag{4.13}\]
Here the constant $C$ will depend on $f$ and $\phi$, but not on $x_0$, $\omega$, $v$, or $\lambda$. Next, since $\psi_1(x)$ is supported on $|x| \geq  {1 \over 2}$, we have 
that $(4.13)$ is bounded by 
\[ C \int_{\{r \in K:\,|\lambda P_{x_0,\omega,v}(r)| \geq {1 \over 2}\}} \frac{r^{n-1}} {|\lambda P_{x_0,\omega,v}(r)|^{n+1}}\,dr\tag{4.14}\]
We split $(4.14)$ dyadically as 
\[ C\sum_{j = 0}^{\infty} \int_{\{r \in K:\,2^{j-1}|\lambda|^{-1} \leq P_{x_0,\omega,v}(r) < 2^j|\lambda|^{-1}\}} \frac{r^{n-1}}{|\lambda  P_{x_0,\omega,v}(r)|^{n+1}}\,dr\tag{4.15}\]
Let $[a_j,b_j]$ denote the $r$ interval of integration in $(4.15)$ corresponding to $j$ and denote its length $b_j  - a_j$ by $l_j$.
Since $\partial_{rr}  P_{x_0,\omega,v}(r) \geq 0$ and therefore $\partial_{r}  P_{x_0,\omega,v}(r)$ is increasing, the lengths of the
 intervals of integration in $(4.15)$ increase no faster than if $ P_{x_0,\omega,v}(r)$ were linear, so that $l_j \leq 2^j l_0$ and $r^{n-1} \leq 2^{j(n-1)} b_0^{n-1}$ on $[a_j,b_j]$. (This can
be seen rigorously by applying the mean-value theorem to $ (P_{x_0,\omega,v})^{-1} (t)$ on the intervals 
$[2^{j-1}|\lambda|^{-1}, 2^j|\lambda|^{-1}]$ for $j \geq 0$ and comparing.)

Thus the 
$j$th term of $(4.15)$ is bounded by $C' 2^{j(n-1)} b_0^{n-1}\times 2^j l_0 \times 2^{-j(n+1)} = C'2^{-j} b_0^{n-1} l_0$. Adding this over all $j$ gives that $(4.15)$ is bounded by
$C''b_0^{n-1} l_0$, which is bounded by a constant times the $j = 0$ term of $(4.15)$, where one has $|\lambda  P_{x_0,\omega,v}(r)| \sim 1$. As a result, we have that $(4.15)$ is bounded
by
\[C'''\int_{\{r \in K:\, |\lambda  P_{x_0,\omega,v}(r)| \leq 1\}} r^{n-1}\,dr \tag{4.16}\]
Going back from polar to retangular coordinates and integrating $(4.16)$ in the $\omega$ variables  leads to a bound of $C'''m(\{x \in \bar{D}: |\lambda| (P(x) \cdot v - P(x_0) \cdot v) < 1\})$.
Since the above measure for $P(x) \cdot v  - P(x_0) \cdot v = P(x) \cdot v  - s(v)$ on $D$ corresponds to, up to a constant factor, the corresponding measure for
 $x \cdot v - s(v)$ on the surface $S'$, and $|\lambda|$ here corresponds to $|t|$ in the statement of Theorem 1.1, this is exactly the desired right hand side of $(1.4)$ and we are done.

\qed

\subsection{The proof of Theorem 1.2.}

We define $E = \{x \in \bar{D}:\,g_i(x) < \,0 {\rm\,\, for\,\, all\,\,} i \}$ and our goal is to bound $I(\lambda)$ given in $(1.5)$ by
\[I(\lambda) = \int_E e^{ - i\lambda_1 x_1 - ... -  i\lambda_nx_n -i\lambda_{n+1}f(x_1,...,x_n)} \phi(x_1,...,x_n)\,dx_1...\,dx_n\tag{4.17}\]
Once again we let  $v = {\lambda \over |\lambda|}$, and since $|I(\lambda)| = |I(-\lambda)|$, once again we may assume $v_{n+1} \geq 0$ as in the statement of the theorem.
Once again we subtract $s(v)$ from the phase, where $s(v)  = \min_{x \in \bar{D}} P(x) \cdot v$, and $x_0$ is any point such that $s(v) = P(x_0) \cdot v$. 
So like $(4.5)$ we have
\[|I(\lambda)| = \bigg|\int_{\R^n} e^{-i|\lambda| (P(x) \cdot v -  P(x_0) \cdot v)}\,\phi(x)\,dx\bigg| \tag{4.18}\]
Next, we let $A = \{x \in D: |\lambda (P(x) \cdot v -  P(x_0) \cdot v)| < 1 \}$ and $2A$ its $x_0$-centered double $\{x_0 + 2(x - x_0): x \in A\}$. We let $\chi_1(x)$ be
 the characteristic function of $2A$ and 
$\chi_2(x) = 1 - \chi_1(x)$, the characteristic function of $(2A)^c$. Denote the integral in $(4.18)$ by $I_0(\lambda)$. We  write $I_0(\lambda) = I_1(\lambda) + I_2(\lambda)$,
 where for $j = 1, 2$ we have
\[I_j(\lambda) = \int_{\R^n}  \chi_j(x) e^{-i|\lambda| (P(x) \cdot v -  P(x_0) \cdot v)}\,\phi(x)\,dx \tag{4.19}\]
We bound $|I_1(\lambda)|$ simply by taking absolute values inside the integrand of $(4.19)$ and integrating. We see that $|I_1(\lambda)|$ is bounded by the 
first term on the right-hand side of $(1.6)$, recalling that $P(x_0) \cdot v = s(v)$ and that the measure of a subset of $S'$ is comparable to the measure of its
projection onto the first $n-1$ coordinates. Hence we devote our attention to bounding $|I_2(\lambda)|$.

In $(4.19)$ for $j = 2$, we switch to polar coordinates, and as in the proof of Theorem 1.1 we define
\[P_{x_0,\omega,v}(r) = P(x_0 + r\omega) \cdot v -  P(x_0) \cdot v\]
Analogous to $(4.6)$ we have
\[I(\lambda) = \int_{S^{n-1}} \int_{E_{\omega, x_0}} e^{-i|\lambda| P_{x_0,\omega,v}(r)}r^{n-1} \phi(x_0 + r\omega)\,dr\,d\omega \tag{4.20}\]
Here $E_{\omega, x_0}$ is the cross section of $E$ being integrated over in $r$. We claim that there is a fixed $N$ such that $E_{\omega, x_0}$ is comprised of at most $N$ intervals for
every $\omega$ and $x_0$. To see this, we locally apply Lemma 4.2 to the functions $g_i(x_0 + r\omega)$, where $r$ corresponds to the $x_n$ variable in Lemma 4.2 and the $x_0$ and
$\omega$ variables correspond to the $x_i$ variables for $i < n$; a compactness argument then gives the statement over all $x_0$, $r$, and $\omega$. Note that the fact that we
 are restricting to $\bar{D}$ does not matter since $\bar{D}$ is convex and won't increase the number of intervals. Similarly the fact that we are restricting to the points where
$|\lambda| P_{x_0,\omega,v}(r) > 1$ doesn't increase the number of intervals since this function is increasing. 

The domain of integration of the inner integral of $(4.20)$, if nonempty, is a collection of at most $N$ intervals. On a given interval, which we call $J$, we integrate by parts as we did in
the proof of Theorem 1.1, but only once this time. The function  $\partial_r P_{x_0,\omega,v}(r)$ will never be zero on $J$ since for $I_2(\lambda)$ we are on the set where 
$|\lambda| P_{x_0,\omega,v}(r) > 1$; since $P_{x_0,\omega,v}(0) = 0$ and $\partial_r P_{x_0,\omega,v}(0)$ is nonnegative and increasing, the latter due to the convexity of $f(x)$, 
we have that  if $|\lambda| P_{x_0,\omega,v}(r) > 1$ then $\partial_r P_{x_0,\omega,v}(r) > 0$. Thus we may proceed as in the proof of Theorem 1.1 and write 
 \[e^{-i|\lambda| P_{x_0,\omega,v}(r) } =  -i|\lambda| \partial_r P_{x_0,\omega,v}(r) e^{-i|\lambda| P_{x_0,\omega,v}(r) } \times (-i|\lambda| \partial_r P_{x_0,\omega,v}(r))^{-1}\]
 In the inner integral of $(4.20)$
 we integrate  $-i|\lambda| \partial_r P_{x_0,\omega,v}(r) e^{-i|\lambda| P_{x_0,\omega,v}(r) }$ back to $e^{-i|\lambda| P_{x_0,\omega,v}(r) }$, and differentiate 
$ (-i|\lambda| \partial_r P_{x_0,\omega,v}(r))^{-1}$ times the remaining factors. We also get two endpoint terms this time. If we denote the endpoints of $J$ by $\alpha_J$ and $\beta_J$,
then the terms obtained at these endpoints are of magnitude bounded by $C (|\lambda| \partial_r P_{x_0,\omega,v}(\alpha_J))^{-1} (\alpha_J)^{n-1}$ and 
$C (|\lambda| \partial_r P_{x_0,\omega,v}(\beta_J))^{-1} (\beta_J)^{n-1}$ respectively. 

We now look at the several places the $r$ derivative may land in this integration by parts and the effect it has. In every case we will simply take absolute values of the resulting integrand 
and bound the resulting integral. First, suppose the derivative lands on the $r^{n-1}$ factor. Then we incur a factor of $Cr^{-1}$, and the absolute value of the resulting term is at most
\[C \int_J  (|\lambda| r \partial_r P_{x_0,\omega,v}(r))^{-1} r^{n-1}\,dr \tag{4.21}\]
Next, we consider the case where the derivative lands on the $(-i|\lambda| \partial_r P_{x_0,\omega,v}(r))^{-1}$ factor. The resulting term is at most
\[C \int_J(|\lambda| (\partial_r P_{x_0,\omega,v}(r))^{-1})^2 (|\lambda| \partial_{rr} P_{x_0,\omega,v}(r))r^{n-1}\,dr\tag{4.22}\]
Here we are using that $\partial_{rr} P_{x_0,\omega,v}(r) \geq 0$ due to the convexity of $f(x)$. In $(4.22)$ we integrate by parts again, this time integrating the $(|\lambda| (\partial_r P_{x_0,\omega,v}(r))^{-1})^2 (|\lambda| \partial_{rr} P_{x_0,\omega,v}(r))$ factor back to $(|\lambda| \partial_r P_{x_0,\omega,v}(r))^{-1}$. We get one integral term bounded by $(4.21)$, 
and two endpoint terms that like before are bounded by $C (|\lambda| \partial_r P_{x_0,\omega,v}(\alpha_J))^{-1} (\alpha_J)^{n-1}$ and 
$C (|\lambda| \partial_r P_{x_0,\omega,v}(\beta_J))^{-1} (\beta_J)^{n-1}$.

Lastly, the $r$ derivative may land on the $\phi(x_0 + r\omega)$ factor. In this case one incurs at most a constant factor, which is better than the $Cr^{-1}$ factor incurred when the 
derivative lands on the $r^{n-1}$ factor. Hence once again $(4.21)$ serves as a bound for the term in question.

Next, we bound the above endpoint terms in terms of an integral resembling $(4.21)$. Namely, since $\partial_r P_{x_0,\omega,v}(r)$ is increasing, we have 
\[(|\lambda| \partial_r P_{x_0,\omega,v}(\alpha_J))^{-1} (\alpha_J)^{n-1} \leq C\int_{{\alpha_J \over 2}}^{\alpha_J} {(|\lambda|r\partial_r P_{x_0,\omega,v}(r))}^{-1}r^{n-1}\,dr\]
\[(|\lambda| \partial_r P_{x_0,\omega,v}(\beta_J))^{-1} (\beta_J)^{n-1} \leq C\int_{{\beta_J \over 2}}^{\beta_J} {(|\lambda|r\partial_r P_{x_0,\omega,v}(r))}^{-1}r^{n-1}\,dr
\tag{4.23}\]
We see that we have boundedly many terms, all bounded by integrals of the form $(4.21)$ or $(4.23)$. The intervals $J$ are all derived from the characteristic function of 
$(2A)^c$ above $(4.19)$, so any of the intervals in either $(4.21)$ or $(4.23)$ come from the characteristic function of $A^c$. Since $A =  \{x \in D: |\lambda (P(x) \cdot v -  P(x_0) \cdot v)| 
< 1 \}$, this means the intervals all are a subset of the points where  $|\lambda P_{x_0,\omega,v}(r)| \geq  1$. As a result, given that there are boundedly many intervals $J$, the
sum of all of the terms $(4.21)$ and $(4.23)$ bounding our inner integral of $(4.20)$ is bounded by
\[C'\int_{|\lambda P_{x_0,\omega,v}(r)| > 1} {(|\lambda|r\partial_r P_{x_0,\omega,v}(r))}^{-1}r^{n-1}\,dr \tag{4.24}\]
Equation $(4.9)$ holds exactly as in the proof of Theorem 1.1, so $(4.24)$ in turn is bounded by 
\[C''\int_{|\lambda P_{x_0,\omega,v}(r)| > 1} {|\lambda  P_{x_0,\omega,v}(r)|}^{-1}r^{n-1}\,dr \tag{4.25}\]
Thus we have an overall bound for $I_2(\lambda)$ of
\[|I_2(\lambda)| \leq C'' \int_{S^{n-1}}\int_{|\lambda P_{x_0,\omega,v}(r)| > 1} {|\lambda  P_{x_0,\omega,v}(r)|}^{-1}r^{n-1}\,dr\,d\omega\tag{4.26}\]
Converting back into rectangular coordinates, we obtain
\[|I_2(\lambda)| \leq C''' \int_{\{x \in \bar{D}: |\lambda|(P(x) \cdot v  - P(x_0) \cdot v) > 1\}}(|\lambda|(P(x) \cdot v - P(x_0) \cdot v ))^{-1} \,dx \tag{4.27}\]
We recall that $P(x_0) \cdot v = s(v)$ and decompose $(4.27)$ dyadically to obtain a bound 
\[|I_2(\lambda)| \leq C'''' \sum_{j = 1}^{\infty} 2^{-j}m(\{x \in \bar{D}: 2^{j-1}|\lambda|^{-1} \leq (P(x) \cdot v  - s(v)) < 2^j|\lambda|^{-1}\})\tag{4.28}\]
Since the above measure for $P(x) \cdot v  - s(v)$ on $D$ corresponds to, up to a constant factor, the corresponding measure for $x \cdot v - s(v)$ on the surface $S'$, 
and $|\lambda|$ here corresponds to $|t|$ in the statement of Theorem 1.2, we
see that $|I_2(\lambda)|$ is bounded by the infinite series in $(1.6)$. This concludes the proof of Theorem 1.2. 

\qed

\section{References.}

\noindent [BakMVW] J. Bak, D. McMichael, J. Vance, S. Wainger, {\it Fourier transforms of surface area measure on convex surfaces in $\R^3$}, Amer. J. Math.{\bf 111} (1989), no.4,
 633-668.  \parskip = 4pt\baselineskip = 3pt

\noindent [BaGuZhZo] S. Basu, S. Guo, R. Zhang, P. Zorin-Kranich, {\it A stationary set method for estimating oscillatory integrals}, J. Eur. Math. Soc. DOI 10.4171/JEMS/1506

\noindent [BrHoI] L. Brandolini, S. Hoffmann, A. Iosevich, {\it Sharp rate of average decay of the Fourier transform of a bounded set}, Geom. Funct. Anal. {\bf 13} (2003), no. 4, 671-680. 

\noindent [BrIT] L. Brandolini, A. Iosevich, G. Travaglini, {\it Planar convex bodies, Fourier transform, lattice points, and irregularities of distribution}, Trans. Amer. Math. Soc. {\bf 355} (2003), no.9, 3513-3535.

\noindent [BNW] J. Bruna, A. Nagel, and S. Wainger, {\it Convex hypersurfaces and Fourier transforms}, Ann. of Math. (2) {\bf 127} no. 2, (1988), 333--365. 

\noindent [CaCWr] A. Carbery, M. Christ, J. Wright, {\it Multidimensional van der Corput and sublevel set estimates}, J. Amer. Math. Soc. {\bf 12} (1999), no. 4, 981-1015. 

\noindent [CoDiMaMu] M. Cowling, S. Disney, G. Mauceri, D. Muller, {\it Damping oscillatory integrals}, Invent. Math. {\bf 101} (1990), no.2, 237-260.

\noindent [G1] M. Greenblatt, {\it A method for bounding oscillatory integrals in terms of non-oscillatory integrals}, submitted. 

\noindent [G2] M. Greenblatt, {\it Hyperplane integrability conditions and smoothing for Radon transforms}, J. Geom. Anal. {\bf 31} (2021), no.4, 3683-3697.

\noindent [Gr] J. Green, {\it Algebraic stability of oscillatory integral estimates: a calculus for uniform estimates}, preprint.

\noindent [Gre] P. Gressman, {\it Scalar oscillatory integrals in smooth spaces of homogeneous type}, Rev. Mat. Iberoam. {\bf 31} (2015), no. 1, 215–244. 

\noindent [Ha1] G. H. Hardy, {\it On the expression of a number as the sum of two squares},  Quart. J. Math. {\bf 46} (1915), 263-283.

\noindent [Ha2] G. H. Hardy, {\it On Dirichlet's divisor problem}, Proc. London Math. Soc. (2) {\bf 15} (1916) 1-25.

\noindent [Hl1] E. Hlawka, {\it \"{U}ber Integrale auf konvexen K\"{o}rpern. I.}  Monatsh. Math. {\bf 54} (1950), 1-36. 

\noindent [Hl2] E. Hlawka, {\it \"{U}ber Integrale auf konvexen K\"{o}rpern. II.} Monatsh. Math. {\bf 54} (1950), 81-99. 

\noindent [IL] A. Iosevich, E Liflyand, {\it Decay of the Fourier transform, analytic and geometric aspects}, Birkhäuser/Springer, Basel, 2014. xii+222 pp.

\noindent [ISa] A. Iosevich, E. Sawyer, {\it Maximal averages over surfaces},  Adv. Math. {\bf 132} (1997), no. 1, 46-119.

\noindent [ISaSe1] A. Iosevich, E. Sawyer, A. Seeger, {\it Two problems associated with convex finite type domains}, Publ. Mat. {\bf 46}, no. 1 (2002), 153-177.

\noindent [ISaSe2] A. Iosevich, E. Sawyer, A. Seeger, {\it Mean square discrepancy bounds for the number of lattice points in large convex bodies}, J. Anal. Math. {\bf 87} (2002), 209-230.

\noindent [ISaSe3] A. Iosevich, E. Sawyer, A. Seeger, {\it Mean lattice point discrepancy bounds. II. Convex domains in the plane}, J. Anal. Math. {\bf 101} (2007), 25-63. 

\noindent [Kr2] E. Kr\"{a}tzel, {\it Analytische Funktionen in der Zahlentheorie}, [Analytic functions in number theory]
Teubner-Texte zur Mathematik [Teubner Texts in Mathematics], {\bf 139}. B. G. Teubner, Stuttgart, 2000. 288 pp. ISBN: 3-519-00289-2.

\noindent [LiY] X. Li and X. Yang, {\it An improvement on Gauss's circle problem and Dirichlet's divisor problem}, preprint.  arxiv 2308.14859.

\noindent [Mi] D. J. Miller, {\it A preparation theorem for Weierstrass systems}, Trans. Amer. Math. Soc. {\bf 358} (2006), no. 10, 4395-4439.

\noindent [PS] D. H. Phong, E. M. Stein, {\it The Newton polyhedron and oscillatory integral operators}, Acta Mathematica {\bf 179} (1997), 107-152.

\noindent [R1] B. Randol, {\it A lattice point problem I}, Trans. Amer. Math. Soc. {\bf 121} (1966), 257-268.

\noindent [R2] B. Randol, {\it A lattice point problem II}, Trans. Amer. Math. Soc. {\bf 125} (1966), 101-113.

\noindent [ShS] R. Shakarchi, E. Stein, {\it Functional analysis. Introduction to further topics in analysis}. Princeton Lectures in Analysis {\bf 4}. Princeton University Press, Princeton, NJ, 2011. xviii+423 pp. ISBN: 978-0-691-11387-6. 
\vskip 0.5 in

\noindent Department of Mathematics, Statistics, and Computer Science \hfill \break
\noindent University of Illinois at Chicago \hfill \break
\noindent 322 Science and Engineering Offices \hfill \break
\noindent 851 S. Morgan Street \hfill \break
\noindent Chicago, IL 60607-7045 \hfill \break
\noindent greenbla@uic.edu

\end{document}